\documentclass[12pt]{amsart}

\usepackage{amsthm,amssymb,amsfonts,amsmath,enumerate,graphicx,color}
\usepackage{mathrsfs}
\usepackage[margin=1in]{geometry}

\usepackage{hyperref}
	\hypersetup{
	colorlinks=true,
	linkcolor={blue},
	citecolor={blue}
	}

\newcommand{\Pc}{\mathcal{P}}

\newcommand{\Cc}{\mathcal{C}}
\newcommand{\Hc}{\mathcal{H}}

\newcommand{\C}{\mathbb{C}}
\newcommand{\Z}{\mathbb{Z}}
\newcommand{\N}{\mathbb{N}}
\newcommand{\R}{\mathbb{R}}
\newcommand{\Q}{\mathbb{Q}}

\renewcommand{\phi}{\varphi}
\renewcommand{\emptyset}{\varnothing}

\def\c{{\boldsymbol c}}

\def\s{{\boldsymbol s}}
\def\t{{\boldsymbol t}}
\def\u{{\boldsymbol u}}
\def\v{{\boldsymbol v}}

\newcommand\commentout[1]{}

\newcommand{\Znn}{\Z_{\geq 0}}
\newcommand{\Zpos}{\Z_{\geq 1}}

\newtheorem{theorem}{Theorem}[section]
\newtheorem{corollary}[theorem]{Corollary}

\newtheorem{lemma}[theorem]{Lemma}

\theoremstyle{remark}

\newtheorem{remark}[theorem]{Remark}
\theoremstyle{definition}

\newtheorem{question}[theorem]{Question}

\begin{document}

\title{Hilbert Bases and Lecture Hall Partitions}

\author{McCabe Olsen}
\address{Department of Mathematics\\
         University of Kentucky\\
         Lexington, KY 40506--0027}
\email{mccabe.olsen@uky.edu}




\date{\today}

\thanks{The author thanks the American Institute of Mathematics, as this work began at the November 2016 workshop on polyhedral geometry and partition theory. The author thanks his advisor, Benjamin Braun, for helpful comments and suggestions throughout this project. The author also thanks the anonymous referees for reading the manuscript carefully and providing helpful suggestions and comments.}


\begin{abstract}
In the interest of finding the minimum additive generating set for the set of $\s$-lecture hall partitions, we compute the Hilbert bases for the $\s$-lecture hall cones in certain cases. In particular, we determine the Hilbert bases for two well-studied families of sequences, namely the $1\mod k$ sequences and the $\ell$-sequences. Additionally, we provide a characterization of the Hilbert bases for $\u$-generated Gorenstein $\s$-lecture hall cones in low dimensions.   
\end{abstract}

\maketitle


\section{Introduction}

For any $n\in\Zpos$, let $[n]:=\{1,2,\cdots,n\}$. Let $\s=(s_1,s_2,\cdots, s_n)$ be a sequence such that $s_i\in \Zpos$ for each $i$. Given any $\s$-sequence, define the \emph{$\s$-lecture hall partitions} to be the set 
	\[
	L_n^{(\s)}:=\left\{\lambda\in\Z^n \, : \, 0 \leq \frac{\lambda_1}{s_1}\leq \frac{\lambda_2}{s_2}\leq \cdots\leq \frac{\lambda_n}{s_n} \right\}.
	\]
In the case when $\s$ is weakly (or strictly) increasing, $\s$-lecture hall partitions are a refinement of the set of all partitions.
Bousqet-M\'{e}lou and Eriksson first introduced the notion of $\s$-lecture hall partitions in two seminal papers \cite{BME-LHP1,BME-LHP2}, and since then these objects have been vastly studied in various contexts. 
For example, lecture hall partitions give rise to variations and generalizations of classical  partition identities, which are of interest in combinatorial number theory. Lecture hall partitions also give rise to various discrete geometric objects, namely rational cones, lattice polytopes, and rational polytopes. These objects have given rise to interesting Hilbert series and Ehrhart theoretic results leading to generalizations of Eulerian polynomials.  
See the excellent survey of Savage \cite{Savage-LHP-Survey} for an overview of much of this work. 

One question that remains open in general is  the following:

\begin{question}\label{PartitionQuestion}
Can we determine the unique minimal additive generating set for $L_n^{(\s)}$ for an arbitrary $\s$? Are there nontrivial bounds on the cardinality of this set?
\end{question}  	

While this is in general a difficult question to answer, one method is to employ tools from polyhedral geometry. Given a sequence $\s=(s_1,\cdots,s_n)$, we define the \emph{$s$-lecture hall cone} to be the rational, pointed, simplical polyhedral cone given by
	\[
	\Cc_n^{(\s)}:=\left\{\lambda\in\R^n \, : \, 0 \leq \frac{\lambda_1}{s_1}\leq \frac{\lambda_2}{s_2}\leq \cdots\leq \frac{\lambda_n}{s_n} \right\}.
	\]
For a rational, pointed cone $C\in \R^n$, the \emph{Hilbert basis} of $C$ is the unique minimal additive generating set of $C\cap\Z^n$. Noting that $\Cc_n^{(\s)}\cap\Z^n=L_n^{(\s)}$, we can now reformulate Question \ref{PartitionQuestion} in terms of polyhedral geometry.

\begin{question}
Can we determine the Hilbert basis of $\Cc_n^{(\s)}$ for arbitrary $\s$? Can we give nontrivial bounds on the cardinality of this set?
\end{question}	

The reformulation in this question seems fruitful.   
Determining the Hilbert basis for a polyhedral cone allows for the study of the Hilbert series of cone, as well as other algebraic interests such as free resolutions of the defining ideal of the cone. This extension of possible questions and problems indicates that it may be a worthwhile pursuit. 
Additionally, some results on Hilbert bases of lecture hall cones are known. Specifically,  Beck, Braun, K\"oppe, Savage, and Zafeirakopoulos  \cite{BeckEtAl-TriangulationsLHC} show that the elements of the Hilbert basis of $\Cc_n^{(\s)}$ for $\s=(1,2,\cdots,n)$ are naturally indexed by subsets $A\subseteq [n-1]$.
Moreover, these elements are all of degree 1 with respect to a particular grading of $\Cc_n^{(\s)}$ and they show that the numerator of the Hilbert series with respect to this grading is an Eulerian polynomial. 
This motivates looking for a general form for arbitrary $\s$.  

Unfortunately, it is unlikely that there is a general structure for the Hilbert bases of $\s$-lecture hall cones, and it is almost a certainty that no nontrivial bounds on the cardinality exist. 
This can be seen in the simplest case $n=2$. 
Let $\s=(s_1,s_2)$ and notice that we have upper and lower bounds; namely $s_1+1$ forms an upper bound given by enumerating lattice points in the fundamental parallelepiped of $\Cc_n^{(\s)}$ and 3 is a lower bound provided $s_1\geq 2$ (2 is the lower bound if $s_1=1$). 
These bounds are in fact sharp, as the sequence $\s=(s_1,k\cdot s_1+1)$ for any $k\in\Zpos$ gives a cone whose Hilbert basis has cardinality $s_1+1$, whereas the cone for the sequence $\s=(s_1,k\cdot s_1-1)$ for any $k\in\Zpos$ has a Hilbert basis of cardinality 3.   
	
Subsequently, in order to obtain meaningful results, we must place some additional restrictions. Motivated by recent work on lecture hall cones \cite{BeckEtAl-GorensteinLHC,BeckEtAl-TriangulationsLHC}, we restrict to the case of \emph{$\u$-generated Gorenstein $\s$-lecture hall cones} (defined in Section \ref{Preliminaries}). We pose the following question.

\begin{question}\label{QuestionGor}
Can we determine the Hilbert basis of $\Cc_n^{(\s)}$ where $\s$ is an arbitrary $\u$-generated Gorenstein sequence? Can we give the cardinality of the set of Hilbert basis elements, or find nontrivial bounds to this set? 
\end{question}

In this paper, we make progress towards answering Question \ref{QuestionGor}. Section \ref{Preliminaries} is devoted to providing necessary definitions and terminology. 
In Sections \ref{modk} and \ref{lsequences}, we consider well-studied families of sequences, namely the $1\,\operatorname{mod}\, k$ sequences and the $\ell$-sequences. In particular, we having the following descriptions of the Hilbert bases:

\begin{theorem}[Theorem \ref{modkthm}]
For all $k\geq 1$, the Hilbert basis for the $1\,\operatorname{mod}\, k$ cones in $\R^n$, denoted $\Cc_{k,n}$, consist of the following elements:
	\begin{itemize}
	\item The element $v_A:=(0,0,\cdots,0, a_1,a_2,\cdots,a_k,a_k+1)$ for each $A\subseteq[n-2]$ where\\  $A=\{a_1<a_2<\cdots<a_k\}$;
	\item Element $w\in L_{k,n}$, where $w_{n-1}=(n-2)k+1$ and $w_n=(n-1)k+1$;  
	\end{itemize}
where $L_{n,k}$ denotes the set of $1\,\operatorname{mod}\, k$ lecture hall partitions. 
\end{theorem}

\begin{theorem}[Theorem \ref{lsequencethm}]
Let $\s=(s_1,s_2,\cdots, s_n)$ be an $\ell$-sequence for some $\ell\geq 2$. The Hilbert basis $\Hc_n^\ell$ for the $\ell$-sequence cone  $\Cc_n^\ell$ is 

	 \[
	 \Hc_n^\ell=\bigcup_{i=0}^{n}\left\{\lambda\in L_n^\ell \ : \ \lambda_{n-1}=s_i \, , \, \lambda_n=s_{i+1}\right\}
	 \]
where $L_n^\ell$ denotes the set of $\ell$-sequence lecture hall partitions.  	 
	\end{theorem}
The necessary definitions and  terminology used in these theorems appear in greater detail in Sections \ref{Preliminaries}, \ref{modk}, and \ref{lsequences}. These main results provide two different generalizations of known Hilbert basis results, as both the $1\,\operatorname{mod}\, k$ sequences and $\ell$-sequences specialize to the sequence  $\s=(1,2,\cdots,n)$ for $k=1$ and $\ell=2$.

In Sections \ref{two}, \ref{three}, and \ref{four}, we provide a  characterization for the Hilbert bases of $\u$-generated Gorenstein $\s$-lecture hall cones in $\R^n$ for $n\leq 4$, noting that the complexity of the Hilbert bases grows rapidly as the dimension increases.  
We conclude the paper in Section \ref{Concluding} by providing some direction for future work in the context of commutative algebra, particularly the study of toric ideals and free resolutions.    

\section{Preliminaries}\label{Preliminaries}

We recall a few definitions from polyhedral geometry. 
A \emph{polyhedral cone} $C$ in $\R^n$ is the solution set to a finite collection of linear inequalities $Ax\geq 0$ for some real matrix $A$, or equivalently for some elements $w_1,w_2,\cdots,w_j \in \R^n$, 
	\[
	C=\operatorname{span}_{\R_{\geq 0}}\{w_1,w_2,\cdots,w_j  \}.
	\]
The elements $w_i$ are called \emph{ray generators}. The cone $C$ is said to be \emph{rational} 	if the matrix $A$ contains only rational entries (equivalently if each $w_i\in\Q^n$), it is said to be \emph{simplicial} if it is defined by $n$ independent inequalities (equivalently if $j=n$ and $\{w_i\}_{i=1}^n$ are linearly independent),  and it is said to be \emph{pointed} if it does not contain a linear subspace of $\R^n$. Let $C^\circ$ denote the interior of $C$. 

Given any pointed rational cone $C\subset\R^n$, a \emph{proper grading} of $C$ is a function $g:C\cap\Z^n\to \Znn^r$, for some $r$, satisfying (i) $g(\lambda+\mu)=g(\lambda)+g(\mu)$; (ii) $g(\lambda)=0$ implies $\lambda=0$; and (iii) for any $v\in \N^r$, $g^{-1}(v)$ is finite. Moreover, the integer points $C\cap \Z^n$ form a semigroup. 
Semigroups of this type have unique minimal generating sets known as the \emph{Hilbert basis} of $C$. Additionally, pointed rational cones give rise to a semigroup algebra structure $\C[C]:=\C[C\cap\Z^n]$.
For background and details see \cite{BeckRobins-CCD,MillerSturmfels-CCA}.

We say that a pointed, rational cone $C\subset\R^n$ is  \emph{Gorenstein} if there exists a point $\c\in~C^\circ\cap~\Z^n$ such that $C^\circ \cap \Z^n=\c+(C\cap \Z^n)$. This point is known as the \emph{Gorenstein point} of $C$. Due to theorems of Stanley \cite{Stanley-HilbertFunctions}, this notion of Gorenstein is equivalent to the commutative algebra notion of Gorenstein, as $C$ is Gorenstein if and only if the algebra $\C[C]$ is Gorenstein. For reference and commutative algebra details, see \cite{BrunsHerzog,StanleyGreenBook}.

It will also be useful to recall several definitions for convex polytopes and Ehrhart Theory.
Let $\Pc\subset \R^n$ be a $n$-dimensional convex polytope with vertex set $\{v_1,v_2,\cdots,v_d\}$. 
We say $\Pc$ is a \emph{lattice polytope} if $v_i\in \Z^n$ for each $i$. 
Likewise, we say that $\Pc$ is a \emph{rational polytope} if $v_i\in\Q^n$ for each $i$.
The \emph{lattice point enumerator} of $\Pc$ is the function 
	\[
	i(\Pc,t)=\# (t\cdot\Pc\cap \Z^n)
	\]
where $t\cdot \Pc=\{t\cdot \alpha \ : \ \alpha\in\Pc\}$ is the $t$th dilate of $\Pc$ with $t\in\Zpos$.  By theorems of Ehrhart \cite{Ehrhart}, if $\Pc$ is lattice, $i(\Pc,t)$ is a polynomial in the variable $t$ of degree $n$ and  if $\Pc$ is rational, $i(\Pc,t)$ is a quasipolynomial in the variable $t$ of degree $d$. Subsequently, we will call $i(\Pc,t)$ the \emph{Ehrhart polynomial} of $\Pc$ or the \emph{Ehrhart quasipolynomial} of $\Pc$ in each respective case. For reference and background on Ehrhart Theory, see \cite{BeckRobins-CCD,Stanley-EC1}.

Given a sequence $\s=(s_1,\cdots,s_n)$, the \emph{$s$-lecture hall cone} is the rational, pointed, simplical polyhedral cone defined as 
	\[
	\Cc_n^{(\s)}:=\left\{\lambda\in\R^n \, : \, 0 \leq \frac{\lambda_1}{s_1}\leq \frac{\lambda_2}{s_2}\leq \cdots\leq \frac{\lambda_n}{s_n} \right\}.
	\]
Alternatively, one may consider a ray generator description with integral generators
	\[
	\Cc_n^{(\s)}=\operatorname{span}_{\R_{\geq 0}}\{(0,\cdots,0,1), (0,\cdots,0,s_i,s_{i+1},\cdots, s_{n-1},s_n) \ : \ 1\leq i \leq n-1\}.
	\] 
It is easy to see that $\Cc_n^{(\s)}\cap \Z^n=L_n^{(\s)}$. 
There are many choices for properly grading the $L_n^{(\s)}$, though three useful notions are as follows:
	\begin{itemize}
	\item $\lambda\mapsto(\lambda_1,\lambda_2,\cdots,\lambda_n)$;
	\item $\lambda\mapsto \lambda_n$;
	\item $\lambda\mapsto (\lambda_n-\lambda_{n-1})$.
	\end{itemize}
In a similar manner, we define the \emph{$\s$-lecture hall polytope} to be 
	\[
	P_n^{(\s)}:=\left\{\lambda\in\R^n \, : \, 0 \leq \frac{\lambda_1}{s_1}\leq \frac{\lambda_2}{s_2}\leq \cdots\leq \frac{\lambda_n}{s_n}\leq 1 \right\}.
	\]
A related geometric structure is the \emph{rational $\s$-lecture hall polytope}, which is defined similarly:
	\[
	R_n^{(\s)}:=\left\{\lambda\in\R^n \, : \, 0 \leq \frac{\lambda_1}{s_1}\leq \frac{\lambda_2}{s_2}\leq \cdots\leq \frac{\lambda_n}{s_n}\leq \frac{1}{s_n} \right\}.
	\]
	
\begin{remark}\label{relativelyprime}
For a given lecture hall cone $\Cc_n^{(\s)}$, we may assume that $\gcd(s_1,s_2,\cdots,s_n)=1$. If we have $\gcd(s_1,s_2,\cdots,s_n)=m>1$, we could consider the sequence $\t=(t_1,\cdots,t_n)$ defined by $t_i=s_i/m$ and notice that it is clear by definition that $\Cc_n^{(\s)}=\Cc_n^{(\t)}$. However, when considering the lecture hall polytope $P_n^{(\s)}$ or the rational lecture hall polytope $R_n^{(\s)}$, it is not permissible to make this assumption. Given two rational polytopes $\mathcal{P},\mathcal{Q}\in\R^n$, we say $\mathcal{P}\cong\mathcal{Q}$ if $\mathcal{Q}= f_U(\mathcal{P})+\v$ where $f_U$ is the linear transformation defined by a unimodular matrix $U$ and $\v\in\R^n$. Note that  $P_n^{(\s)}\not\cong P_n^{(\t)}$ and $R_n^{(\s)}\not\cong R_n^{(\t)}$.  In fact, we have $P_n^{(\s)}=m\cdot P_n^{(\t)}$ and  $R_n^{(\s)}=m\cdot R_n^{(\t)}$.  
\end{remark}

There has been much study of these three polyhedral geometric objects (see, e.g., \cite{BeckEtAl-GorensteinLHC,BeckEtAl-TriangulationsLHC,PensylSavage-Wreath,PensylSavage-Rational,Savage-LHP-Survey,SavageViswanathan-1/kEulerian}).
In particular, a characterization of which $\s$-sequences yield Gorenstein cones was implicitly given by Bousquet-M\'{e}lou and Eriksson in \cite{BME-LHP2} and explicitly stated by Beck, Braun, K\"oppe, Savage, and Zafeirakopoulos in \cite{BeckEtAl-GorensteinLHC} as follows: 

\begin{theorem}[Beck et al {\cite[Corollary 2.6]{BeckEtAl-GorensteinLHC}}, Bousquet-M\'{e}lou, Eriksson {\cite[Proposition 5.4]{BME-LHP2}}] \label{GorensteinCriteria}
For a positive integer sequence $\s$, the $\s$-lecture hall  cone $\Cc_n^{(\s)}$ is Gorenstein if and only if there exists some $\c \in \Z^n$ satisfying
	\[
	c_j s_{j-1}=c_{j-1}s_j+\gcd(s_j,s_{j+1})
	\]
for $j>1$, with $c_1=1$. 	
\end{theorem}

Moreover, in the case of $\s$-sequences where $\gcd(s_i,s_{i+1})=1$ holds for all $i$, we have a refinement to this theorem. We say that $\s$ is \emph{$\u$-generated} by a sequence $\u$ of positive integers if $s_2=u_1s_1-1$ and $s_{i+1}=u_is_i-s_{i-1}$ for $i>1$. 

\begin{theorem}[Beck et al {\cite[Theorem 2.8]{BeckEtAl-GorensteinLHC}}, Bousquet-M\'{e}lou, Eriksson {\cite[Proposition 5.5]{BME-LHP2}}]\label{PairwiseGorenstein}
Let $\s=(s_1,\cdots,s_n)$ be a sequence of positive integers  such that $\gcd(s_i,s_{i+1})=1$ for $1\leq i<n$. Then $\Cc_n^{(\s)}$ is Gorenstein if and only if $\s$ is $\u$-generated by some sequence $\u=(u_1,u_2,\cdots,u_{n-1})$ of positive integers. When such a sequence exists, the Gorenstein point $\c$ for $\Cc_n^{(\s)}$ is defined by $c_1=1$, $c_2=u_1$, and for $2\leq i<n$, $c_{i+1}=u_ic_i-c_{i-1}$.
\end{theorem}

It is a natural question to consider the Hilbert basis of a given polyhedral cone. While the question of characterizing the Hilbert bases for $\Cc_n^{(\s)}$ for arbitrary $\s$ is intractible, a natural redirection is to restrict to the case of $\u$-generated Gorenstein $\s$-sequences. To provide further motiviation, Beck, Braun, K\"oppe, Savage, and Zafeirakopoulos in \cite{BeckEtAl-TriangulationsLHC} give an explicit description of the Hilbert basis in the case of $\s=(1,2,\cdots,n)$, which is $\u$-generated by $\u=(3,2,2,\cdots,2)$. The Hilbert basis is given as follows.

\begin{theorem}[Beck et al {\cite[Theorem 5.1]{BeckEtAl-TriangulationsLHC}}]\label{BBKSZHilbert}
For each $A=\{a_1<a_2<\cdots<a_k\}\subseteq[n-1]$, define the element $v_A$ to be
	\[
	v_A=(0,\cdots,0, a_1,a_2,\cdots, a_l, a_k+1).
	\]	
The Hilbert basis for $L_n^{(1,2,\cdots,n)}$ is 
	\[
	\Hc_n^{(1,2,\cdots,n)}:=\{v_A \ : \ A \subseteq [n-1]\}.
	\]
As a corollary, the semigroup algebra $\C[\Cc_n^{(1,2,\cdots,n)}]$ is generated entirely by elements in degree $1$ with respect to the grading given by $\lambda\mapsto(\lambda_n-\lambda_{n-1})$. 
\end{theorem}


\section{The $1\,\operatorname{mod}\, k$ sequences}\label{modk}

For any $k\in \Z_{\geq 1}$, we define the $1\,\operatorname{mod}\, k$ sequence to be 
	\[
	\s=(1,k+1,2k+1,\cdots,(n-1)k+1).
	\]
For convenience of notation, let $L_{k,n}:=L_n^{(s)}$, let $\Cc_{k,n}:=\Cc_n^{(s)}$, and let $P_{k,n}:=P_n^{(s)}$. 
This sequence is $\u$-generated by $\u=(k+2,2,2,\cdots,2)$, and hence Gorenstein. 
Note that if $k=1$, we obtain the sequence $\s=(1,2,\cdots,n)$. 
This generalization has been well studied, most notably by Savage and Viswanathan \cite{SavageViswanathan-1/kEulerian} using a discrete geometric point of view. 
We now give a concise description for the Hilbert basis of $\Cc_{k,n}$.

\begin{theorem}\label{modkthm}
For all $k\geq 1$, the Hilbert basis $\Hc_{k,n}$ of $L_{k,n}$ consists of the following elements:
	\begin{itemize}
	\item The element $v_A:=(0,0,\cdots,0, a_1,a_2,\cdots,a_k,a_k+1)$ for each $A\subseteq[n-2]$ where\\  $A=\{a_1<a_2<\cdots<a_k\}$.
	\item Element $w\in L_{k,n}$, where $w_{n-1}=(n-2)k+1$ and $w_n=(n-1)k+1$.  
	\end{itemize}
\end{theorem}

\begin{proof}
The Hilbert basis for the case of $k=1$ is known by Theorem \ref{BBKSZHilbert} and  the description can be translated to be written in this language with ease. Subsequently, we will prove the result assuming $k\geq 2$.

First we claim that $v_A$ are all possible elements of degree one with respect to the grading given by $\deg(\lambda)=\lambda_n-\lambda_{n-1}$. Let $a=(a_1,a_2,\cdots,a_{n-1},a_n)\in L_{k,n}$ such that $a_n-a_{n-1}=1$ and $a_{n-1}<n-1$. 
We can see that $a=v_A$ for some set $A$ for the following reasons:
	\begin{enumerate}[(i)]
	\item For each $1\leq i\leq n-1$, $a_{n-1}<n-1$ implies that $a_{n-i}<n-i$ because we have the inequlities
	\[
	\frac{n-i}{k(n-i)+1}\leq \frac{n-i+1}{k(n-i+1)+1}\leq \cdots\leq \frac{n-3}{k(n-3)+1}\leq \frac{n-2}{k(n-2)+1}
	\]
but we also clearly have 
	\[
	\frac{n-i+1}{k(n-i)+1}\not\leq \frac{n-i+1}{k(n-i+1)+1};
	\]
	\item  We must have $a_i<a_{i+1}$ for all $i\leq n-1$ as the inequlities 
	\[
	\frac{a_{i+1}-1}{k(i-1)+1}<\frac{a_{i+1}}{ki+1}
	\]
are equivalent to $a_{i+1}\leq i$, but we also clearly have
	\[
	\frac{a_{i+1}}{k(i-1)+1}\not\leq\frac{a_{i+1}}{ki+1}.
	\]	
	\end{enumerate}
Hence, we have $a=(0,\cdots,0,a_j,a_{j+1},\cdots,a_{n-1},a_{n-1}+1)$ which means $a=v_A$ for the set $A=\{a_j<a_{j+1}<\cdots<a_{n-1}\}\subset[n-2]$.
Now suppose that $a\in L_{k,n}$ and suppose that $a_{n-1}=j\geq n-1$. 
Notice that $a_n\geq j+2$, because if we suppose that $a_n=j+1$, then we arrive at a contradiction as 
	\[
	\frac{j}{k(n-2)+1}\leq \frac{j+1}{k(n-1)+1}
	\]
holds if and only if $j<(n-1)$, which violates hypothesis. Therefore, $a$ must be of degree 2 or higher.

Second, note that $w\in L_{k,n}$, with $w_{n-1}=(n-2)k+1$ and $w_n=(n-1)k+1$ cannot be written as a combination of elements of the type $v_A$. This follows from a grading argument as $w$ has degree $k$. If we consider $a=\sum_{i=1}^k v_{A_i}$, it is clear that $a_{n-1}\leq k(n-2)<k(n-2)+1=w_{n-1}$ and we have the result. 

Now, suppose that $a\in L_{k,n}$. There are three possible cases:
	\begin{enumerate}
	\item $a_{n-1}<k(n-2)+1$ and $a_n<k(n-1)+1$;
	\item $a_{n-1}<k(n-2)+1$ and $a_n\geq k(n-1)+1$;
	\item $a_{n-1}\geq k(n-2)+1$ and $a_n\geq k(n-1)+1$.	
	\end{enumerate}		 

Case 1:  Suppose that $a_{n-1}<k(n-2)+1$ and $a_n<k(n-1)+1$. 
Given that $s_1=1$, this condition forces $a_1=0$, because  
	\[
	a_1\leq \frac{a_n}{k(n-1)+1}<1
	\]
and likewise for all $2\leq i\leq n-2$ we have $a_i<k(i-1)+1$ because	
	\[
	\frac{a_i}{k(i-1)+1}<1.
	\]
Moreover, we note that for all such $i$, we have
	\[
	\frac{a_i}{k(i-1)+1}<\frac{a_{i+1}}{ki+1}
	\]
because equality would force 
	\[
	a_{i+1}=a_i+ k\cdot \frac{a_i}{k(i-1)+1}
	\]
which cannot be an integer by our previous observation and the fact that $\gcd(k,k(i-1)+1)=1$.		
Let $j$ be the largest index such that $a_{j}<j-1$. We now write $a=b+c$ where
	\[
	b=(0,a_2,\cdots,a_j,j,j+1,\cdots, n-1)
	\] 
and 
	\[
	c=(0,0,\cdots,0,a_{j+1}-j, a_{j+2}-(j+1),\cdots,a_n-(n-1)).
	\]
It is clear that $b=v_A$ for some $A\subseteq [n-2]$. To show that $c\in L_{k,n}$, notice that for all $i\geq j$ 
	\[
	\frac{a_i-i+1}{k(i-1)+1}\leq \frac{a_{i+1}-1}{ki+1}
	\]	
is equivalent to  
	\[
	a_i(ki+1)+1\leq a_{i+1}(k(i-1)+1)
	\]	
which is equivalent to  
	\[
	\frac{a_i}{k(i-1)+1}<\frac{a_{i+1}}{ki+1}
	\]	
and thus we have the desired result. So by induction, $a$ of this form can be written as the sum of elements of the type $v_A$. 

Case 2: Suppose that $a_{n-1}<k(n-2)+1$ and $a_n\geq k(n-1)+1$. We claim that $a-v_{\emptyset}=a-(0,0,\cdots,0,1)\in L_{k,n}$. If  $a_n>k(n-1)+1$, this is immediate. So, suppose that $a_n=k(n-1)+1$, then
	\[
	\frac{a_{n-1}}{k(n-2)+1}\leq \frac{k(n-2)}{k(n-2)+1}<\frac{k(n-1)}{k(n-1)+1}=\frac{a_n-1}{k(n-1)+1}
	\] 
holds because $k>0$. Thus, for $a$ of this form we can reduce to Case 1.

Case 3: Suppose that $a_{n-1}\geq k(n-2)+1$ and $a_n\geq k(n-1)+1$. Let $j$ be the largest index such that $a_j<k(j-1)+1$. We write $a=b+c$, where 
	\[
	b=(a_1,a_2,\cdots,a_j,kj+1,k(j+1)+1,\cdots, k(n-2)+1,k(n-1)+1)
	\]
and 
	\[
	c=(0,\cdots,0,a_{j+1}-(kj+1),a_{j+2}-(k(j+1)+1),\cdots,a_{n-1}-(k(n-2)+1),a_n-(k(n-1)+1)).
	\]
It is clear that $b\in L_{k,n}$ with $b_{n-1}=k(n-2)+1$ and $b_n=k(n-1)+1$, which is an element of our proposed Hilbert basis. Moreover, because for all $i\geq j$ we have $a_i\geq k(i-1)+1$ by assumption, it is immediate that $c\in L_{k,n}$. Thus, by induction, this case will reduce to either Case 1 or Case 2 showing the result.				
\end{proof}

In addition to the description of the Hilbert basis, we can also give the cardinality of the Hilbert basis by using Ehrhart theoretic methods. 

\begin{corollary}

	\[
	|\Hc_{k,n}|=\frac{(k+1)^{n-2}+(k-1)}{k}+2^{n-2}.
	\]
\end{corollary}
\begin{proof}
Given that we have an element $v_A$ for all $A\subseteq [n-2]$, this yields $2^{n-2}$ elements. To enumerate the remaining Hilbert basis elements, note that there is a clear bijection between $w\in L_{k,n}$ with $w_{n-1}=(n-2)k+1$ and $w_n=(n-1)k+1$, and elements $w'\in L_{k,n-2}$ such that $w'_{n-2}\leq (n-3)k+1$. However, for any such $w'$, one can identify $w'$ as a lattice point in the polytope $P_{k,n-2}$. Savage and Viswanathan \cite[Theorem 2]{SavageViswanathan-1/kEulerian} prove that the Ehrhart polynomial of $P_{k,n}$ is given by
	\[
	i(P_{k,n},t)=(-1)^t\sum_{p=0}^t {{\frac{1}{k}-1} \choose {t-p}} {{-\frac{1}{k}\,} \choose p}(kp+1)^{n}.
	\]
Evaluating $i(P_{k,n-2},t)$ at $t=1$ yields
	\[
	i(P_{k,n-2},1)=(-1)\left(\frac{1}{k}-1\right)+(-1)\left(-\frac{1}{k}(k+1)^{n-2}\right)=\frac{(k+1)^{n-2}+(k-1)}{k}.
	\]	
Thus, the proof is complete. 	
\end{proof}


\section{The $\ell$-sequences}\label{lsequences}

For any $\ell\in\Z_{\geq 2}$, define the $\ell$-sequence to  be $\s=(s_1,s_2,\cdots,s_n)$ recursively as follows: $s_{i+1}=\ell s_i-s_{i-1}$ with $s_0=0$ and $s_1=1$.  
For convenience of notation let $L_n^\ell:=L_n^{(\s)}$, $\Cc_n^\ell:=\Cc_n^{(\s)}$, and $R_n^\ell:=R_n^{(\s)}$. 
Note that it is easy to see that any $\ell$-sequence is strictly increasing. 
Moreover, $\ell$-sequences are $\u$-generated by the sequence $\u=(\ell+1,\ell,\ell,\cdots,\ell)$, and hence $\Cc_n^\ell$ is Gorenstein. 
If we let $\ell=2$, we reduce to the known case of $\s=(1,2,\cdots,n)$. 
The $\ell$-sequences have appeared from a number theoretic point of view by way of the $\ell$-lecture hall theorem and $\ell$-Euler theorems studied in \cite{BME-LHP2} and \cite{SavageYee-lsequences}. 
We now give an explicit description of the Hilbert basis for any $\ell$-sequence lecture hall cone.

\begin{theorem}\label{lsequencethm}
Let $\s=(s_1,s_2,\cdots, s_n)$ be an $\ell$-sequence for some $\ell\geq 2$. The Hilbert basis of $\Cc_n^\ell$ is

	 \[
	 \Hc_n^\ell=\bigcup_{i=0}^{n}\left\{\lambda\in L_n^\ell \ : \ \lambda_{n-1}=s_i \, , \, \lambda_n=s_{i+1}\right\}
	 \]
	\end{theorem}

\begin{proof}
Note that the Hilbert basis for $\ell=2$ is given by Theorem \ref{BBKSZHilbert} and can be translated into this notation with ease. We will use the convention that $s_i=0$ if $i\leq 0$. We first claim that there are no redundancies in $\Hc_n^\ell$.
First note that $w\in L_n^\ell$ with $w_{n-1}=s_2=\ell$ and $w_n=s_3=\ell^2-1$ cannot be written as a combination of smaller elements of the proposed Hilbert basis. 
This is true because it would imply $w=\ell\cdot v'+c\cdot u$ where $v_{n-1}=1$, $v_n=\ell$, $u_{n-1}=0$, and $u_n=1$, but this is contradiction as $w_n=\ell^2+c$ for some positive integer $c$. 
Now, suppose that for some $i\geq 3$ there exists $w\in L_n^\ell$ such that $w_{n-1}=s_i$ and $w_n=s_{i+1}$ with $w=\sum v_j$ where each $v_j$ is an element of the proposed Hilbert basis as well. 
This would imply that 
	\[
	s_i=\sum_{k=1}^ma_k\cdot s_k
	\]
where $a_k\in\Z_{\geq 0}$ and $m<i$ and that 
	\[
	s_{i+1}=\sum_{k=1}^ma_k\cdot s_{k+1}.
	\]	
However, since  $s_{i+1}=\ell\cdot s_i-s_{i-1}$, combining these two gives us that 
	\[
	s_{i-1}=\sum_{k=1}^m a_k\cdot s_{k-1}.
	\]
We can now use this equality along with $s_i=\ell\cdot s_{i-1}-s_{i-2}$ to deduce that
	\[
	s_{i-2}=\sum_{k=1}^m a_k\cdot s_{k-2}.
	\] 
In fact, we can continue this iteration so that 
	\[
	s_{i-j}=\sum_{k=1}^m a_k\cdot s_{k-j}.
	\] 
In the case  $j=i-2$, 
	\[
	s_{2}=\sum_{k=1}^m a_k\cdot s_{k-i+2}=a_{i-1}\cdot s_1=a_{i-1}
	\]
which implies that $m=i-1$ and that $a_m=a_{i-1}=\ell$	 as $s_2=\ell$. However, this implies that 
	\[
	s_i=\ell\cdot s_{i-1}+ \sum_{k=1}^{i-2}a_k\cdot s_k
	\]
with $a_k\in\Znn$, which is a contradiction to $s_{i}=\ell\cdot s_{i-1}-s_{i-2}$. Thus, we have no redundancy.
	
Let $\lambda=(\lambda_1,\lambda_2,\cdots, \lambda_{n-1},\lambda_n)\in L_n^\ell$. First note that if $\lambda_{n-1}\geq s_i$ then $\lambda_n\geq s_{i+1}$. Notice that the inequality
	\[
	\frac{s_i}{s_{n-1}}< \frac{s_{i+1}}{s_n}
	\]
is equivalent to 
	\[
	s_is_n<s_{i+1}s_{n-1},
	\]	
and making the substitutions $s_n=\ell s_{n-1}-s_{n-2}$ and $s_{i+1}=\ell s_i-s_{i-1}$ and simplifying leads to the new equivalent statement
	\[
	s_{i-1}s_{n-1}<s_i s_{n-3}.
	\]
Repeating this process similarly shows that the above inequalities are equivalent to 
	\[
	s_{i-j}s_{n-j}<s_{i-j+1}s_{n-j-1}
	\]	
for any $1\leq j\leq i$. So, if $j=i$, note that $s_0=0$, $s_1=1$, and we have that $0<s_{n-j-1}$ which is necessarily true. Moreover, if $\lambda_n<s_{i+1}$ then we have the inequality 	
	\[
	\frac{s_i}{s_{n-1}}\leq \frac{s_{i+1}-1}{s_n}
	\]
which is equivalent to 
	\[
	s_i s_n\leq s_{i+1}s_{n-1}-s_{n-1}.
	\]
Making similar reductions as above for any $1\leq j\leq i$ this is equivalent to one of the following:
	\[
	\begin{cases}
	s_{i-j}s_{n-j}\leq s_{i-j+1}s_{n-j-1}-s_{n-1}, & \mbox{ if } j \mbox{ is even }\\
	s_{i-j}s_{n-j} +s_{n-1}\leq s_{i-j+1}s_{n-j-1}, & \mbox{ if } j \mbox{ is odd }.\\ 
	\end{cases}
	\]			
If we consider $j=i$, either of the preceding is equivalent to  
	\[
	s_{n-1}\leq s_{n-j-1}
	\]
which is a contradiction because $\s$ is a strictly increasing sequence for any $\ell$. Ergo, $\lambda_{n-1}\geq s_i$ implies $\lambda_n\geq s_{i+1}$.

Consider $\lambda_{n-1}$. If $\lambda_{n-1}\geq s_{n-1}$, we have $\lambda_n\geq s_n$. Let $j$ be the smallest integer such that $\lambda_j\geq s_j$. Notice that $(\lambda_1,\cdots,\lambda_{j-1},s_j,\cdots,s_{n-1},s_n)\in L_n^\ell$ and  $\lambda-(\lambda_1,\cdots,\lambda_{j-1},s_j,\cdots,s_{n-1},s_n) \in L_n^\ell$ follows immediately. 

Now suppose that $s_i\leq \lambda_{n-1}<s_{i+1}$. Notice, since $s_1=1$, that we can write the element $\lambda_{n-1}=k\cdot s_i+\sum_{a_p\in A}s_{a_p}$ where $1\leq k<\ell$, $A$ is a multiset of elements of $[i-1]$ of cardinality $r<\infty$, and each $a_p$ is chosen to be as large as possible. 
Then we have that $\lambda_n> k\cdot s_{i+1}+\sum_{a_p\in A}s_{a_p+1}$. 
This is an elementary exercise akin to the previous proof that $\lambda_{n-1}\geq s_i$ implies $\lambda_n> s_{i+1}$.
 To see that $\lambda-(\lambda_1,\cdots,\lambda_{n-1},s_i,s_{i+1})\in L_n^\ell$, first suppose that  we write $\displaystyle\lambda_{n-1}=k\cdot s_i+\sum_{t=0}^{i-1}b_t \cdot s_t$ where $b_t\in\Z_{\geq 0}$ are the multiplicities of the elements of the multiset described above. 
Now, we have 
	\[
	\displaystyle \frac{\displaystyle\left(k\cdot s_i+\sum_{t=0}^{i-1}b_t\cdot s_t\right) -s_i}{s_{n-1}}\leq \frac{\displaystyle\left(k\cdot s_{i+1}+\sum_{t=0}^{i-1}b_t\cdot s_{t+1}\right) -s_{i+1}}{s_n}\leq \frac{\lambda_n-s_{i+1}}{s_n}.
	\]
The second equality is immediate by previous observation and the first inequality is equivalent to 
	\[
	(k-1)s_is_n+s_n\cdot \sum_{t=0}^{i-1}b_t\cdot s_t \leq (k-1)s_{i+1}s_{n-1}+ s_{n-1}\cdot \sum_{t=0}^{i-1}b_t\cdot s_{t+1}.
	\]	
By expanding using $s_n=\ell\cdot s_{n-1}-s_{n-2}$ on the right hand side and $s_{i+1}=\ell\cdot s_i-s_{i-1}$ and $s_{t+1}=\ell\cdot s_t-s_{t-1}$ on the left hand side, we have after simplification 
	\[
	 (k-1)s_{n-2}s_i+\sum_{t=0}^{i-1}b_t\cdot s_t \geq (k-1)s_{n-1}s_{i-1}+\sum_{t=0}^{i-1}b_t\cdot s_{t-1}.
	\]
In a similar manner to the above, we arrive at the equivalent statement 
	\[
	(k-1)s_{n-j}s_{i-j}+\sum_{t=0}^{i-1}b_t\cdot s_{t-j}\leq (k-1)s_{n-j-1}s_{i-j+1}+\sum_{t=0}^{i-1}b_t\cdot s_{t-j+1}
	\]
for any $0\leq j\leq i$. When $j=i$,  
	\[
	0=(k-1)s_{n-i}s_{0}\leq (k-1)s_{n-i-1}s_1=(k-1)s_{n-i-1}
	\]
which is necessarily true. Therefore, by induction, we have a complete Hilbert basis. 	
\end{proof}

We now provide a method for computing the cardinality of the Hilbert basis for any $\ell$-sequence. Though not given by an explicit algebraic expression, this formula gives a  combinatorial interpretation for the cardinality of the Hilbert basis elements of $\ell$-sequences.  

\begin{corollary}\label{lsequencecardinality}
	\[
	|\Hc_n^\ell|=2+\sum_{j=1}^{n-2}i(R_{n-2}^\ell,s_j)
	\] 
where $i(R_{n-2}^\ell,t)$ denotes the Ehrhart quasipolynomial of the rational lecture hall polytope $R_{n-2}^\ell$.	 
\end{corollary}	

\begin{proof}
Suppose that $\lambda\in L_n^\ell$ such that $\lambda_{n-1}=s_{i+1}$ and $\lambda_n=s_{i+2}$ for some $1 \leq i\leq n-2$. This implies that $\lambda_{n-2}\leq s_i$ by similar applying arguments used in the proof of Theorem \ref{lsequencethm}. Therefore, we can bijectively associate $\lambda$ with a lattice point $\lambda'$ in the $s_i$th dilate of the rational lecture hall polytope $R_{n-2}^\ell$, so $\lambda'\in (s_{i}\cdot R_{n-2}^\ell\cap \Z^{n-2})$. Therefore, all such Hilbert basis elements are enumerated by $i(R_{n-2}^\ell, s_i)$. All Hilbert basis elements are counted in this way with the exception of two, namely $(0,\cdots,0,0,1)$ and $(0,\cdots,0,1,\ell)$, as $s_1=1$ and $s_2=\ell$. Thus, we have the desired.
\end{proof}
As an aside, note that the $(n-2)$th summand of the cardinality expression of Corollary \ref{lsequencecardinality} actually gives $i(R_{n-2}^{\ell},s_{n-2})=i(P^\ell_{n-2},1)$. This means that some of the Hilbert basis elements correspond to lattice points in the integral lecture hall polytope $P^\ell_{n-2}$, which one may have suspected  from the results in the $1\,\operatorname{mod}\, k$ cones. This phenomenon occurs in later cases as well.

\section{Two-dimensional Gorenstein sequences}\label{two}
We begin our low-dimensional characterization for the Hilbert bases of $\u$-generated Gorenstein lecture cones by considering the two-dimensional case.
Notice that when $n=2$, Remark \ref{relativelyprime} implies that there is no distinction between Gorenstein and $\u$-generated Gorenstein. 
Applying Theorems \ref{GorensteinCriteria} and \ref{PairwiseGorenstein} provides the following description for the Gorenstein condition. 

\begin{lemma}\label{2dGorenstein}
Suppose that $\s=(s_1,s_2)$ such that $\Cc_2^{(\s)}$ is Gorenstein. Then $\s=(s_1,ks_1-1)$ for $k\geq 1$.
\end{lemma}

Using this description, we will now classify the Hilbert bases for all two-dimensional Gorenstein lecture hall cones as follows.

\begin{theorem}
Let $\Cc_2^{(\s)}$ be a Gorenstein lecture hall cone with $\s=(s,ks-1)$ for some $k\geq 1$. The Hilbert basis of $\Cc_2^{(\s)}$ is $\Hc_2^{(\s)}=\{(0,1),(s,ks-1),(1,k)\}$. 
\end{theorem} 

\begin{proof}
Let $(a,b)\in L_2^{(\s)}$. 
First, suppose that $a\geq s$ and note that this immediately  implies that $b\geq ks-1$. 
We have that $(a,b)-(s,ks-1)\in L_2^{(\s)}$ because 
	\[
	\frac{a-s}{s}\leq \frac{b-(ks-1)}{ks-1}
	\]
follows directly from 
	\[
	\frac{a}{s}\leq \frac{b}{ks-1}
	\]	
and that $a\geq s$ and $b\geq ks-1$. 

Now suppose that $1\leq a\leq s-1$. 
If $a\geq 1$, then $b\geq k$ because $\frac{1}{s}<\frac{k}{ks-1}$, but $\frac{1}{s}>\frac{k-1}{ks-1}$. 
Observe that 
	\[
	\frac{a}{s}<\frac{b}{ks-1}
	\] 	
must hold, because equality implies that $b=ak-\frac{a}{s}$ which by the assumption $1\leq a\leq s-1$ cannot be an integer. 
Now, we claim that $(a,b)-(1,k)\in L_2^{(\s)}$, as 
	\[
	\frac{a-1}{s}\leq \frac{b-k}{ks-1}
	\]
is equivalent to 
	\[
	a(ks-1)<bs,
	\]	
which is equivalent to our observation above.

Finally, note that if $a=0$ and $b\geq 1$, $(a,b)-(0,1)\in L_2^{(\s)}$ is immediate. 
Thus, by induction, we have a complete Hilbert basis. 	
\end{proof}

We note that when $n=2$, the Gorenstein condition ensures that the Hilbert basis is of the smallest possible cardinality, $|\Hc_2^{(\s
)}|=3$ when $s_1\geq 2$ and $|\Hc_2^{(\s)}|=2$ if $s_1=1$. 
This further motivates the restriction to $\u$-generated Gorenstein cones. 

\section{Three-dimensional $\u$-generated Gorenstein sequences}\label{three}
We continue our low dimensional characterization for $\u$-generated Gorenstein lecture hall cones by considering the three-dimensional case. 
When $n=3$, a direct application of Theorem \ref{PairwiseGorenstein} yields the following description.

\begin{lemma}
Suppose that $\s=(s_1,s_2,s_3)$ such that $\Cc_3^{(\s)}$ is Gorenstein with $\gcd(s_i,s_{i+1})=1$ for all $i$. Then $\s=(s,ks-1,\ell(ks-1)-s)$ for integers $s\geq 1$, $k\geq 1$ and $\ell\geq 1$.
\end{lemma}

Using the above lemma, we now completely characterize the Hilbert bases for all $\u$-generated Gorenstein lecture hall cones for $n=3$.

\begin{theorem}\label{3dGor}
Suppose that $\s=(s,ks-1,\ell(ks-1)-s)$. Then 
	\begin{itemize}
	\item If $s\geq 2$, then the Hilbert basis is \[ \Hc_3^{(\s)}=\{(0,0,1),(0,1,\ell),(0,k,\ell k-1),(1, k,\ell k-1),(j,ks-1,\ell(ks-1)-s) \ \forall \ 0\leq j\leq s\}. \]
	\item If $s=1$, then the Hilbert basis is 
	\[	\Hc_3^{(\s)}=\{(0,0,1),(0,1,\ell), (0,k-1,\ell(k-1)-1),(1,k-1,\ell(k-1)-1)\}. \] 
	\end{itemize}

\end{theorem}
\begin{proof}
Suppose that $s\geq 2$. We claim that the proposed Hilbert basis has no redundancy. 
It is sufficient to show that $(j,ks-1,\ell(ks-1)-s)$ cannot be written as a sum of other proposed elements. 
Suppose this is possible, then there exist positive integers $\alpha,\beta$, and $\gamma$ such that  $\alpha(k)+\beta=ks-1$. 
This has solutions $\alpha=s-i$ and $\beta=ki-1$ for $1\leq i<s$. 
However, we must also have $\alpha(\ell k-1)+\beta(\ell)+\gamma=\ell(ks-1)-s$ and evaluating at the above solution implies that $\gamma=-i$. 
This is a contradiction.     

Let $(a,b,c)\in L_3^{(\s)}$. First note that if $a\geq s$, this implies that $b\geq ks-1$ and $c\geq \ell(ks-1)-s$. It is clear then that $(a,b,c)-(s,ks-1,\ell(ks-1)-s)\in L_3^{(\s)}$. If $0\leq a<s$ and $b\geq ks-1$, then it follows that $c\geq   c\geq \ell(ks-1)-s$ and $(a,b,c)-(a,ks-1,\ell(ks-1)-s)\in L_3^{(\s)}$.

Next suppose that $1\leq a<s$ and $b<ks-1$. Notice that $a\geq 1$ implies that $b\geq k$ and $c\geq \ell k-1$ because $\frac{1}{s}<\frac{k}{ks-1}<\frac{\ell k-1}{\ell(ks-1)-s}$ but $\frac{1}{s}>\frac{k-1}{ks-1}$ and $\frac{k}{ks-1}>\frac{\ell k-2}{\ell(ks-1)-s}$. Additionally, we can see that the inequalities must be strict: 
	\[
	\frac{a}{s}<\frac{b}{ks-1}<\frac{c}{\ell(ks-1)-s}.
	\]
This follows because equality of the first and second fractions implies that $b=ak-\frac{a}{s}$ which is not an integer by the assumption $1\leq a<s-1$, and  equality of second and third fractions implies that $c=b\ell-\frac{bs}{ks-1}$ which is not an integer by the assumption $b<ks-1$.
Now, we claim that $(a,b,c)-(1,k,\ell k-1)\in L_3^{(\s)}$, as
	\[
	\frac{a-1}{s}\leq \frac{b-k}{ks-1}\leq \frac{c-\ell k+1}{\ell(ks-1)-s}
	\]
is equivalent to 
	\[
	a(ks-1)+1\leq bs \ \mbox{  and  } \  b(\ell(ks-1)-s)\leq c(ks-1)-1
	\]	
or	
	\[
	a(ks-1)< bs \ \mbox{    and    } \  b(\ell(ks-1)-s)< c(ks-1)
	\]	
which is equivalent to the strict inequalities shown above.	

Now suppose that $a=0$. If $b\geq k$, we have $(0,b,c)-(0,k,\ell k-1)\in L_3^{(\s)}$ immediately by the previous argument. So, suppose that $1\leq b<k$, and notice that this implies that $c\geq b\ell$ as $\frac{b}{ks-1}<\frac{b\ell}{\ell(ks-1)-s}$. However, we also have $\frac{b}{ks-1}>\frac{b\ell-1}{\ell(ks-1)-s}$ as this is equivalent to $sb<ks-1$ which follows from $b\leq k-1$. We now claim that $(0,b,c)-(0,1 \ell)\in L_3^{(\s)}$ as we have the following inequalities
	\[
	\frac{b-1}{ks-1}\leq\frac{b\ell-\ell}{\ell(ks-1)-s}\leq \frac{c-\ell}{\ell(ks-1)-s}.
	\]
The second inequality is immediate by $c\geq b\ell$ and the first inequality is equivalent to $b\geq 1$.  	

Thus, by induction, any element of $L_3^{(\s)}$ can be written as a sum of these elements and we have the Hilbert basis.

Now, we suppose that $s=1$. It is clear that there is no redundancy in the proposed Hilbert basis. Note that we must have $k\geq 2$. Let $(a,b,c)\in L_3^{(\s)}$. Consider $b$. If $b\geq k-1$, then $c\geq \ell(k-1)-1$. If $a\geq 1$, then $(a,b,c)-(1,k-1,\ell(k-1)-1)\in L_3^{(\s)}$ is immediate. If $a=0$, then $(a,b,c)-(0,k-1,\ell(k-1)-1)\in L_3^{(\s)}$ is also immediate. Now, if $1\leq b <k-1$, note that $a=0$ and $c\geq b\ell$, which follows from the same argument given in the previous case. Moreover, we also have $(a,b,c)-(0,1,\ell)\in L_3^{(\s)}$ immediately from work of the previous case. Thus, by induction, we have the Hilbert basis.    
\end{proof}
 
We note that in this case, the cardinality of the Hilbert basis is directly dependent on the starting value $s_1$, with $|\Hc_3^{(\s)}|=s_1+5$ when $s_1\geq 2$ and $|\Hc_3^{(\s)}|=4$ when $s_1=1$. 
 
\section{Four-dimensional $\u$-generated Gorenstein sequences}\label{four}

We conclude our low-dimensional characterization of $\u$-generated Goresntein lecture hall cones in the  case of four dimensions. We have the following description for the Hilbert bases.

\begin{theorem}\label{4dCharacterization}
Suppose that $\s=(s_1,s_2,s_3,s_4)$ is  $\u$-generated by $\u=(u_1,u_2,u_3)$ such that $\Cc_4^{(\s)}$ is a Gorenstein lecture hall cone. Recall that $\c=(c_1,c_2,c_3,c_4)$ is the Gorenstein point of $\Cc_4^{(\s)}$, with $c_1=1$, $c_2=u_1$, and $c_{i+1}=u_ic_i-c_{i-1}$ for $i\geq 2$. Then
	\begin{enumerate}[(a)]
	\item If $s_1=1$ and $u_1=2 $ and the Hilbert basis is
		\[
		\Hc_4^{(\s)}=\{ (0,0,0,1),(0,0,1,u_3),(0,0,s_3,s_4),(0,1,s_3,s_4),(1,1,s_3,s_4)\}.
		\]
	\item If $s_1=1$ and $u_1\geq 3$ and the Hilbert basis is
		\[
		\Hc_4^{(\s)}=\left.\begin{cases}
		(0,j,s_3,s_4) \mbox{ for all } 0\leq j\leq s_2\\
		(1,s_2,s_3,s_4),		
		(0,0,0,1),
		(0,0,1,u_3),
		(0,0,u_2,u_2u_3-1),
		(0,1,u_2,u_2u_3-1)
		\end{cases}\right\}.
		\]
	\item If $s_1=2$ and $u_1=1$, then the Hilbert basis is 
		\[
		\Hc_4^{(\s)}=\left\{(2,1,s_3,s_4),(1,1,s_3,s_4),(0,1,s_3,s_4),(0,0,s_3,s_4),(0,0,1,u_3), (0,0,0,1) \right\}.
		\]
	\item If $s_1\geq 3$ and $u_1=1$, then the  Hilbert basis is
		\[
		\Hc_4^{(\s)}=\left. \begin{cases}
		\lambda\in L_4^{(\s)} \mbox{ with }   \lambda_3=s_3 \mbox{ and } \lambda_4=s_4\\
		(0,0,0,1),(0,0,1,u_3),(0,0,c_3,c_4),(0,1,c_3,c_4),(1,1,c_3,c_4)\\
		
		\end{cases}\right\}.
		\]
	\item If $s_1\geq 2$ and $u_1\geq 2$, then the Hilbert basis is
		\[
		\Hc_4^{(\s)}= \left.\begin{cases}
		\lambda\in L_4^{(\s)} \mbox{ with }   \lambda_3=s_3 \mbox{ and } \lambda_4=s_4\\
		(0,j,c_3,c_4) \mbox{ for all }  0\leq j\leq c_2\\
		(c_1,c_2,c_3,c_4),
		(0,0,0,1),
		(0,0,1,u_3),
		(0,0,u_2,u_2u_3-1),
		(0,1,u_2,u_2u_3-1)
		\end{cases}\right\}.
		\]
		
	\end{enumerate}	 
\end{theorem}

\begin{proof} For each of the cases, we will consider $\Cc_4^{(\s)}$ with respect to the grading defined by  $\lambda\mapsto(\lambda_4-\lambda_3)$. The first two cases (a) and (b) can be reduced to the three dimensional case and hence follow directly from the proof of Theorem \ref{3dGor}.

\textbf{Case (c):} First note that to have a valid sequence $u_2\geq 3$ and  $s_2=1$, $s_3=u_2-2$, and $s_4=u_3(u_2-2)-1$. It is clear that there are no redundancies among the elements of the proposed Hilbert basis.

We will now show that an arbitrary element of $\L_4^{(\s)}$ can be written as a sum of elements of this basis by induction. 
Let $\lambda\in L_4^{(\s)}$ and consider $\lambda_3$. If $\lambda_3\geq s_3$, we then consider $\lambda_2=0$ or $\lambda_2\geq 1$. If $\lambda_2=0$, it is clear that $\lambda-(0,0,s_3,s_4)\in L_4^{(\s)}$. If  $\lambda_2\geq 1$, we then consider $\lambda_1=0$, $\lambda_1=1$, or $\lambda_1\geq 2$. We can see that if $\lambda_1=0$, then $\lambda-(0,1,s_3,s_4)\in L_4^{(\s)}$ and if $\lambda_1=1$, then  $\lambda-(1,1,s_3,s_4)\in L_4^{(\s)}$. If $\lambda_1\geq 2$, note that $\lambda_2\geq \lceil \frac{\lambda_1}{2}\rceil$ which implies $\lambda-(2,1,s_3,s_4)\in L_4^{(\s)}$ follows from 
	\[
	\frac{\lambda_1-2}{2}\leq \left\lceil\frac{\lambda_1}{2}\right\rceil-1\leq \lambda_2-1
	\]
which is clearly true. 

Now, suppose that $1\leq \lambda_3< s_3$. This implies that $\lambda_1=\lambda_2=0$ and $\lambda_4\geq u_3\lambda_3$. The first two are trivial and the latter follows because 
	\[
	\frac{\lambda_3}{u_2-2}<\frac{u_3\lambda_3}{u_3(u_2-2)-1}
	\]
holds, but the inequality
	\[
	\frac{\lambda_3}{u_2-2}\leq \frac{u_3\lambda_3-1}{u_3(u_2-2)-1}
	\]
is equivalent to $\lambda_3\geq s_3$ which is a contradiction. Moreover, we have that $\lambda-(0,0,1,u_3)\in L_4^{(\s)}$ because 
	\[
	\frac{\lambda_3-1}{u_2-2}\leq \frac{u_3\lambda_3-u_3}{u_3(u_2-2)-1}\leq \frac{\lambda_4-u_3}{u_3(u_2-2)-1}
	\]	
is immediate from the previous observations. Therefore, by induction, we have a complete Hilbert basis.	
\bigskip

\textbf{Case (d):} First, we claim that this set contains no redundancy. Note that no element $\lambda\in L_4^{(\s)}$ with $\lambda_3=c_3$ and $\lambda_4=c_4$ can be written as a combination of smaller elements. Given that $c_3=u_2-1$ and $c_4=u_3u_2-u_3-1$, this would imply that $c_4=(u_2-1)u_3+b$ where $b\in \Z_{\geq 0}$ which is impossible. Suppose that $w\in L_4^{(\s)}$ such that $w_3=s_3$, $w_4=s_4$ and there are additional elements of the proposed Hilbert basis such that $\sum_{i=1}^dv_i=w$. Note that this would imply there  are integers $m,n,p\in \Z_{\geq 0}$, where $m\leq s_1-2$ so that we have $\displaystyle s_3=m\cdot c_3+n=m\cdot u_2-m+n$ and $\displaystyle s_4=m\cdot c_4+n\cdot u_3+p$. However, we also have $s_4=u_3s_3-s_1+1$. Combining and simplification yields the result $s_1=m-p-1$, which is a contradiction to $m\leq s_1-2$, and hence no such sum exists.    

We will now show that an arbitrary element of $\L_4^{(\s)}$ can be written as a sum of elements of this basis by induction. Suppose that $\lambda\in L_4^{(\s)}$ and consider $\lambda_3$. If $\lambda_3\geq s_3$, then consider $\lambda_2$ and $\lambda_1$. One of three cases will hold (i) $\lambda<s_2$, (ii) $\lambda_2\geq s_2$ with $\lambda_1 <s_1$, or (iii) $\lambda_1\geq s_1$. For (i), it is clear that $\lambda-(\lambda_1,\lambda_2,s_3,s_4)\in L_4^{(\s)}$, for (ii) it is clear that $\lambda-(\lambda_1,s_2,s_3,s_4)\in L_4^{(\s)}$ and for (iii) it is clear that $\lambda-(s_1,s_2,s_3,s_4)\in L_4^{(\s)}$ all of which are valid lecture hall partitions.

Now, suppose that $c_3\leq \lambda_3<s_3$. We can then write $\lambda_3=\alpha\cdot c_3+\beta$ where either $1\leq \alpha<s_1-2$ and $0\leq \beta \leq c_3-1=u_2-2$ or $\alpha=s_1-2$ and $0\leq \beta\leq u_2-3$, because $s_3=(s_1-2)c_3+(u_2-2)$. We claim that $\lambda_4\geq \alpha\cdot c_4+\beta\cdot u_3$. This follows because 
	\[
	\frac{\alpha\cdot c_3+\beta}{s_3}<\frac{\alpha\cdot c_4+\beta\cdot u_3}{s_4}
	\]
reduces using $c_3=u_2-1$, $c_4=u_3\cdot (u_2-1)-1$, $s_3=u_2s_1-u_2-s_1$, and $	s_4=u_3u_2s_1-u_3s_1-u_2u_3-s_1+1$ to the inequality  
	\[
	-\alpha+\beta(1-s_1)<0,
	\]
which is obviously true. However, the inequality 
	\[
	\frac{\alpha\cdot c_3+\beta}{s_3}\leq\frac{\alpha\cdot c_4+\beta\cdot u_3-1}{s_4}
	\]
reduces in the same manner to 
	\[
	\alpha+\beta(s_1-1)\geq u_2s_1-s_1-u_2,
	\]
which is contradiction because of the conditions 1 that $\alpha\leq s_1-3$ and $\beta\leq u_2-2$ or $\alpha=s_1-2$ and $\beta\leq u_2-3$. 

Now consider $\lambda_2$. Suppose that the $1\leq \lambda_2<s_1-1=s_2$, then $\lambda_3\geq c_3\lambda_2$. This follows because 
	\[
	\frac{\lambda_2}{s_2}<\frac{\lambda_2(u_2-1)}{s_3}
	\] 
is equivalent to $\lambda_2>0$, but the  inequality
	\[
	\frac{\lambda_2}{s_2}\leq \frac{\lambda_2(u_2-1)-1}{s_3}
	\] 
cannot hold because it reduces to $s_1-1\leq \lambda_2$, which is a contradiction.	

If $\lambda_2\geq 1$ we have
	\[
	\frac{\lambda_2-1}{s_2}\leq \frac{\lambda_2\cdot c_3-c_3}{s_3}\leq \frac{\lambda_3-c_3}{s_3}=\frac{\alpha\cdot c_3 +\beta-c_3}{s_3}\leq \frac{\alpha\cdot c_4 +\beta\cdot u_3-c_4}{s_4}\leq \frac{\lambda_4-c_4}{s_4}.
	\]	
Each of these inequalities follows directly from previous observations. It now holds that if $\lambda_1=0$, we have that $\lambda-(0,0,c_3,c_4)\in L_4^{(\s)}$ in the case $\lambda_2=0$ and $\lambda- (0,1,c_3,c_4)\in L_4^{(\s)}$ provided $\lambda_2>0$.

Moreover, if $1\leq\lambda_2<s_1-1$ and $\lambda_1\geq 1$, then $\frac{\lambda_1}{s_1}<\frac{\lambda_2}{s_1-1}$ as equality creates a contradiction. This yields the equivalent inequality
	\[
	\frac{\lambda_1-1}{s_1}\leq \frac{\lambda_2-1}{s_1-1},
	\]
which implies  $\lambda-(1,1,c_3,c_4)\in L_4^{(\s)}$.

Now, suppose that $1\leq \lambda_3<c_3$. Notice that this implies $\lambda_1=\lambda_2=0$ and $\lambda_4\geq u_3\lambda_3$. The first two inequalities are immediate, and the latter inequality follows from the fact that 
	\[
	\frac{\lambda_3}{s_3}<\frac{\lambda_3\cdot u_3}{s_4}	
	\]
is equivalent to $s_1>1$, which is true by assumption, but the inequality
	\[
	\frac{\lambda_3}{s_3}\leq\frac{\lambda_3\cdot u_3-1}{s_4}
	\]	
using the observation $c_3=u_2-1$, reduces to 
	\[
	u_2s_1-u_2-s_1\leq \lambda_3(s_1-1)\leq (u_2-2)(s_1-1)=u_2s_1-u_2-2s_1+2,
	\]	
which contradicts the assumption that $s_1\geq 3$.	
Subsequently, we have $\lambda-(0,0,1,u_3)\in L_4^{(\s)}$ by the above observation and applying the arguments used in case (c). Thus, by induction, we have a complete Hilbert basis in this case. 	 			
\bigskip

\textbf{Case (e):} 
We verify that the proposed set contains no redundancy. 
It is clear that the elements $(0,0,0,1)$, $(0,0,1,u_3)$, $(0,0,u_2,u_2u_3-1)$, and $(0,0,u_2,u_2u_3-1)$ cannot be written as a combination of one another.
So suppose first that an element $(0,j,c_3,c_4)=\sum_{i=1}^e v_i$ where the $v_i$ are elements of smaller degree.
This implies that there exist $a,b,d\in \Z_{\geq 0}$ such that $a\cdot u_2+b=c_3$ and $a\cdot(u_2u_3-1)+b\cdot u_3+d=c_4$, with the restriction that $a\leq u_1-1$ because $c_3=u_1u_2-1$.
However, we also have $c_4=u_3c_3-c_2=u_3c_3-u_1$, which means that  $a-d=u_1$, which contradicts $a\leq u_1-1$.
Thus, these elements cannot be written as a sum of elements of lower degree.   

Now, suppose that $w\in L_4^{(\s)}$ with $w_3=s_3$ and $w_4=s_4$. 
If there was some collection of elements lower degree in the proposed basis such that $w=\sum_{i=1}^e v_i$, this implies the existence of integers $m,n,p,q\in \Z_{\geq 0}$ with $m\leq s_1-1$, $n\leq u_1-2$ such that $s_3=m\cdot c_3+n\cdot u_2+p$, $s_4=m\cdot c_4+n(u_2u_3-1)+p\cdot u_3+q$, but also $s_4=u_3s_3-s_2$. 
When we combine and simplify, we have 
 $s_2=m\cdot c_2+n-q$. 
However, this implies that $u_1s_1-1\leq (u_1-2)(s_1-1)-q$, which implies that $u_1+s_1\leq 3-q$, which contradicts $s_1\geq 2$ and $u_1\geq 2$. 
Hence, there are no redundancies in the proposed Hilbert basis. 

We will now show that an arbitrary element of $L_4^{(\s)}$ can be written as a sum of elements of this basis by induction. 
Let $\lambda\in L_4^{(\s)}$ and consider $\lambda_3$. 
If $\lambda_3\geq s_3$, we can construct an element of $w\in L_4^{(\s)}$ such that $w_3=s_3$ and $w_4=s_4$ so that $\lambda-w\in L_4^{(\s)}$ by following analogous construction to the previous cases. 

If $c_3\leq\lambda_3<s_3$, note that then we can consider $\lambda_1$. If $\lambda_1=0$, we have that $\lambda-(0,\lambda_2,c_3,c_4)\in L_4^{(\s)}$ provided that $\lambda_2\leq c_2$ or $\lambda-(0,c_2,c_3,c_4)\in L_4^{(\s)}$ in the case $\lambda_2>c_2$. If $\lambda_1\geq 1$, we have $\lambda-(c_1,c_2,c_3,c_4)\in L_4^{(\s)}$. Each of these statements follow identically from the arguments for the $c_3\leq \lambda_3<s_3$ made in case (d).

Now, suppose that $u_2\leq \lambda_3 <c_3=u_2u_1-1$. We can write $\lambda_3=\alpha\cdot u_2 +\beta	$ where either $1\leq\alpha\leq u_1-2$ and $0\leq\beta\leq u_2-1$ or $\alpha=u_1-1$ and $0\leq\beta\leq u_2-2$, as $c_3=(u_1-1)u_2+(u_2-1)$. Note that this implies that $\lambda_4\geq \alpha(u_2u_3-1)+\beta\cdot u_3$ because the inequality 
	\[
	\frac{\alpha\cdot u_2 +\beta}{s_3}<\frac{\alpha(u_2u_3-1)+\beta\cdot u_3}{s_4}
	\]
reduces to $0<\alpha\cdot s_1+\beta(u_1s_1-1)$ which is true by assumption. Additionally, the inequality
	\[
	\frac{\alpha\cdot u_2 +\beta}{s_3}\leq\frac{\alpha(u_2u_3-1)+\beta\cdot u_3-1}{s_4}
	\]
reduces to 
	\[
	u_2u_1s_1-u_2-s_1\leq \alpha s_1+\beta (u_1s_1-1)
	\]	
due to the assumptions on $\alpha$ and $\beta$ implies that either
	\[
	u_2u_1s_1-u_2-s_1\leq (u_1-2) s_1+(u_2-1) (u_1s_1-1)=(u_2u_1s_1-u_2-s_1)-s_1+1,
	\]	
which contradicts $s_1\geq 2$, or it implies 
	\[
	u_2u_1s_1-u_2-s_1\leq (u_1-1) s_1+(u_2-2) (u_1s_1-1)=(u_2u_1s_1-u_2-s_1)-u_1s_1+2,
	\]
which 	contradicts $s_1\geq 2$ and $u_1\geq 2$. Moreover, note that this implies that $\lambda_2<u_1$, which implies that $\lambda_1=0$. Additionally, we have that $1\leq \lambda_2 < u_1$ implies that $\lambda_3\geq \lambda_2 u_2$, which follows because the inequality 
	\[
	\frac{\lambda_2}{s_2}=\frac{\lambda_2}{s_1u_1-1}<\frac{\lambda_2u_2}{u_2u_1s_1-u_2-s_1}=\frac{\lambda_2u_2}{s_3}
	\] 	
is immediate, but the inequality
	\[
	\frac{\lambda_2}{s_1u_1-1}<\frac{\lambda_2u_2-1}{u_2u_1s_1-u_2-s_1}
	\]	
reduces to $u_1s_1-1\leq \lambda_2s_1$ which contradicts $\lambda_2<u_1$. Therefore, we get the inequalities
	\[
	\frac{\lambda_2-1}{s_2}\leq \frac{\lambda_2u_2-u_2}{s_3}\leq \frac{\lambda_3-u_2}{s_3}
	\]
and 
	\[
	\frac{\lambda_3-u_2}{s_3}=\frac{(\alpha-1)\cdot u_2+\beta}{s_3}\leq \frac{(\alpha-1)\cdot (u_2u_3-1)+\beta u_3 }{s_4}\leq \frac{\lambda_4-(u_2u_3-1)}{s_4}.
	\]		
Thus, if we have $\lambda-(0,1,u_2,u_2u_3-1)\in L_4^{(\s)}$ when $\lambda_2\geq 1$ and $\lambda-(0,0,u_2,u_2u_3-1)\in L_4^{(\s)}$ when $\lambda_2=0$.

If $1\leq \lambda_3 <u_2$, we get $\lambda-(0,0,1,u_3)\in L_4^{(\s)}$ by repeating analogous arguments to previous cases (see case (c)). Thus, by induction, we have a complete Hilbert basis. 
\end{proof}

Given the explicit Hilbert basis in the case of $n=4$, it is additionally of interest to consider the  cardinalities of the set in each case. The following is computation of the these cardinalities.

\begin{corollary}
For each case of Theorem \ref{4dCharacterization}, the cardinality of the Hilbert basis is as follows:
	\begin{enumerate}[(a)]
	\item If $s_1=1$ and $u_1=2$, then $|\Hc_4^{(\s)}|=5$.
	\item If $s_1=1$ and $u_1\geq 3$, then $|\Hc_4^{(\s)}|=s_2+6$.
	\item If $s_1=2$ and $u_1=1$, then $|\Hc_4^{(\s)}|=6$.
	\item If $s_1\geq 3$ and $u_1=1$, then $\displaystyle |\Hc_4^{(\s)}|=\frac{(s_1+1)(s_1-2)}{2}+5$.
	\item  If $s_1\geq 2$ and $u_1\geq 2$, then $\displaystyle |\Hc_4^{(\s)}|=\frac{u_1(s_1(s_1+1))}{2}+u_1^2+6$.
	\end{enumerate}
\end{corollary}
\begin{proof}
The cases of (a), (b), and (c) are immediate. Consider case (d). It is necessary to enumerate the number of $\lambda\in L_4^{(\s)}$ such that $\lambda_3=s_3$ and $\lambda_4=s_4$. We should notice that this equivalent to determining the number of lattice points in a lecture hall polytope, namely  $P_2^{(s_1,s_2)}$. Given that this is a lattice triangle, this is an easy task. Note that $s_2=s_1-1$ and the vertices of $P_2^{(s_1,s_2)}$ are $(0,0)$, $(0,s_1-1)$, and $(s_1,s_1-1)$.  Recall Pick's theorem says that if $\Pc$ is a lattice polygon with area $A$, $I$ interior lattice points, and $B$ boundary lattice points, then 
	\[
	A=I+\frac{B}{2}-1
	\]
must hold (see \cite{BeckRobins-CCD} for details and proof).
We can see that there are $2s_1$ lattice points on the boundary of $P_2^{(s_1,s_2)}$ as the hypotenuse contains only the vertices as lattice points by $\gcd(s_1,s_2)=1$. Moreover, since the area is $\frac{s_1(s_1-1)}{2}$, we get that there are $ \frac{s_1^2-3s_1-2}{2}$ interior lattice points. Adding the interior points, the boundary points, and the additional five Hilbert basis elements gives
	\[
	\frac{s_1^2-3s_1-2}{2} +2s_1+5=\frac{(s_1+1)(s_1-2)}{2}+5.
	\] 
	
To show (e), we apply similar methods. We must enumerate the lattice points of $P_2^{(s_1,s_2)}$, where $s_2=u_1s_1-1$, which has vertices $(0,0),(0,u_1s_1-1)$, and $(s_1,u_1s_1-1)$. We find that there are $s_1(u_1+1)$ boundary points, again noting that the hypotenuse contains only the two vertices. Applying Pick's theorem, yields that there are $\frac{u_1s_1^2-s_1(u_1+2)+2}{2}$ interior points. Hence, we have that  $P_2^{(s_1,s_2)}$ contains $\frac{u_1(s_1(s_1+1)}{2}+1$ lattice points. Additionally, elements of the form $(0,j,c_3,c_4)$ account for $c_2+1=(u_1^2-1)+1=u_1^2$ elements, and there are 5 additional described elements. This gives the cardinality desired.  	
\end{proof}

\section{Concluding remarks and future directions}\label{Concluding}

It is possible that one could consider continuing the  low-dimensional characterization to $n=5$ or greater dimensions. 
However, there are two observations, which discourage this pursuit. 
First, as noted by the case of $n=4$, as the dimension increases so does the complexity and variation of the Hilbert basis. 
Experimental evidence using Normaliz \cite{Normaliz} indicates that there would be many more cases to consider in the case of $n=5$, and this will likely shroud the  significance of knowing the Hilbert bases.  
Secondly, cardinality arguments are unlikely in general for greater dimension. 
The cardinality of the Hilbert basis is controlled in large part by the first $n-2$ terms of the $\s$-sequence.
In particular, it appears that to obtain the cardinality of the Hilbert basis, one must always compute the number of lattice points in $P_{n-2}^{(\s)}$.
In the case of $n=4$, $n-2=2$ and Pick's theorem makes this possible, but there is no analogue to Pick's theorem for dimension $\geq 3$, which makes the task much more difficult.

There are a number of different directions for future research in this vein. 
To begin, one could consider the computation of Hilbert bases for more families of $\s$-lecture hall cones. 
One particular family of well studied sequences which fall under the umbrella of $\u$-generated Gorenstein sequence are the \emph{$(k,\ell)$-sequences} (see \cite[Section 5]{Savage-LHP-Survey} for definition and importance). 
However, it is certainly possible that some sequences yield lecture hall cones with combinatorially interesting Hilbert bases that are not $\u$-generated Gorenstein, or even Gorenstein. 
It may be interesting to consider certain sequences of this type (e.g., the Fibonacci sequence).

In addition, knowing the Hilbert bases for $\s$-lecture hall cones opens the door for a number of questions of an algebraic flavor. 
It is well known that $\C[\Cc_n^{(\s)}]\cong \C[x_1,\cdots,x_d]/I_\s$  where $d=|\Hc_n^{(\s)}|$ and $I_\s$ is a toric ideal. 
It would be of interest to compute these toric ideals in certain cases to determine if these ideals admit algebraically or combinatorially interesting Gr\"obner bases under certain term orders, as well as other algebraic or algebro-geometric properties.  
Moreover, one could consider free resolutions of $I_\s$ to determine the multigraded Betti numbers of the $\Cc_n^{(\s)}$, either using algebraic or combinatorial methods \cite{Sturmfel-GrobnerPolytopes}. 
These are unknown even in the case  $\s=(1,2,\cdots,n)$. 


\bibliographystyle{plain}
\bibliography{newbib}{}

\begin{thebibliography}{10}

\bibitem{BeckEtAl-GorensteinLHC}
Matthias Beck, Benjamin Braun, Matthias K{\"o}ppe, Carla~D. Savage, and
  Zafeirakis Zafeirakopoulos.
\newblock s-lecture hall partitions, self-reciprocal polynomials, and
  {G}orenstein cones.
\newblock {\em Ramanujan J.}, 36(1-2):123--147, 2015.

\bibitem{BeckEtAl-TriangulationsLHC}
Matthias Beck, Benjamin Braun, Matthias K{\"o}ppe, Carla~D. Savage, and
  Zafeirakis Zafeirakopoulos.
\newblock Generating functions and triangulations for lecture hall cones.
\newblock {\em SIAM J. Discrete Math.}, 30(3):1470--1479, 2016.

\bibitem{BeckRobins-CCD}
Matthias Beck and Sinai Robins.
\newblock {\em Computing the continuous discretely: Integer-point enumeration
  in polyhedra}.
\newblock Undergraduate Texts in Mathematics. Springer, New York, 2007.

\bibitem{BME-LHP1}
Mireille Bousquet-M\'elou and Kimmo Eriksson.
\newblock Lecture hall partitions.
\newblock {\em Ramanujan J.}, 1(1):101--111, 1997.

\bibitem{BME-LHP2}
Mireille Bousquet-M\'elou and Kimmo Eriksson.
\newblock Lecture hall partitions. {II}.
\newblock {\em Ramanujan J.}, 1(2):165--185, 1997.

\bibitem{Normaliz}
W.~Bruns, B.~Ichim, T.~R{\"o}mer, R.~Sieg, and C.~S{\"o}ger.
\newblock Normaliz. {A}lgorithms for rational cones and affine monoids.
\newblock Available at \texttt{https://www.normaliz.uni-osnabrueck.de}.

\bibitem{BrunsHerzog}
Winfried Bruns and J\"urgen Herzog.
\newblock {\em Cohen-{M}acaulay rings}, volume~39 of {\em Cambridge Studies in
  Advanced Mathematics}.
\newblock Cambridge University Press, Cambridge, 1993.

\bibitem{Ehrhart}
Eug\`ene Ehrhart.
\newblock Sur les poly\`edres homoth\'etiques bord\'es \`a {$n$} dimensions.
\newblock {\em C. R. Acad. Sci. Paris}, 254:988--990, 1962.

\bibitem{MillerSturmfels-CCA}
Ezra Miller and Bernd Sturmfels.
\newblock {\em Combinatorial commutative algebra}, volume 227 of {\em Graduate
  Texts in Mathematics}.
\newblock Springer-Verlag, New York, 2005.

\bibitem{PensylSavage-Wreath}
Thomas~W. Pensyl and Carla~D. Savage.
\newblock Lecture hall partitions and the wreath products {$C_k\wr S_n$}.
\newblock {\em Integers}, 12B(Proceedings of the Integers Conference
  2011):Paper No. A10, 18, 2012/13.

\bibitem{PensylSavage-Rational}
Thomas~W. Pensyl and Carla~D. Savage.
\newblock Rational lecture hall polytopes and inflated {E}ulerian polynomials.
\newblock {\em Ramanujan J.}, 31(1-2):97--114, 2013.

\bibitem{Savage-LHP-Survey}
Carla~D. Savage.
\newblock The mathematics of lecture hall partitions.
\newblock {\em J. Combin. Theory Ser. A}, 144:443--475, 2016.

\bibitem{SavageViswanathan-1/kEulerian}
Carla~D. Savage and Gopal Viswanathan.
\newblock The {$1/k$}-{E}ulerian polynomials.
\newblock {\em Electron. J. Combin.}, 19(1):Paper 9, 21, 2012.

\bibitem{SavageYee-lsequences}
Carla~D. Savage and Ae~Ja Yee.
\newblock Euler's partition theorem and the combinatorics of
  {$\ell$}-sequences.
\newblock {\em J. Combin. Theory Ser. A}, 115(6):967--996, 2008.

\bibitem{Stanley-HilbertFunctions}
Richard~P. Stanley.
\newblock Hilbert functions of graded algebras.
\newblock {\em Advances in Math.}, 28(1):57--83, 1978.

\bibitem{StanleyGreenBook}
Richard~P. Stanley.
\newblock {\em Combinatorics and commutative algebra}, volume~41 of {\em
  Progress in Mathematics}.
\newblock Birkh\"auser Boston, Inc., Boston, MA, second edition, 1996.

\bibitem{Stanley-EC1}
Richard~P. Stanley.
\newblock {\em Enumerative combinatorics. {V}olume 1}, volume~49 of {\em
  Cambridge Studies in Advanced Mathematics}.
\newblock Cambridge University Press, Cambridge, second edition, 2012.

\bibitem{Sturmfel-GrobnerPolytopes}
Bernd Sturmfels.
\newblock {\em Gr\"obner bases and convex polytopes}, volume~8 of {\em
  University Lecture Series}.
\newblock American Mathematical Society, Providence, RI, 1996.

\end{thebibliography}

\end{document}